\documentclass[12pt,twoside, reqno]{article}
\usepackage{amsmath}
\usepackage{amssymb}
\usepackage{graphicx}

\usepackage{mathrsfs}
\usepackage{pb-diagram}

\def\eop{\hfill$\square$}

\newtheorem{theorem}{Theorem}[section]
\newtheorem{lemma}{Lemma}[section]

\numberwithin{equation}{section}



\begin{document}

\title{Period two implies chaos for a class of ODEs: a dynamical system approach}

\begin{author}
{Marina Pireddu \footnote{Department of Mathematics for the Decisions, University of Florence}} 
\end{author}

\date{}
\maketitle

\begin{abstract}
\noindent The aim of this note is to set in the field of dynamical systems a recent theorem by Obersnel and Omari in 
\cite{ObOm-07} about the presence of periodic solutions of all periods for a class of scalar
time-periodic first order differential equations without uniqueness, provided a subharmonic solution (and 
thus, for instance, a solution of period two) does exist.
Indeed, making use of the Bebutov flow, we try to clarify in what sense the term ``chaos'' has to be understood and 
which dynamical features can be inferred for the system under analysis.  
\end{abstract}

\section{Introduction and Motivation}\label{sec-1}
In the recent papers \cite{ObOm-04,ObOm-07} Obersnel and Omari and in \cite{DCObOm-06} De Coster, Obersnel and Omari, using upper and lower solutions techniques, give a complete description of the structure of the set of solutions of the scalar time-periodic first order differential equation
\begin{equation}\label{eq-fo}
\dot x = f(t,x),
\end{equation}
where $f: {\mathbb R} \times {\mathbb R}\to {\mathbb R}$ satisfies the $L^1$-Carath\'eodory conditions and is
$1$-periodic in the time-variable, that is, $f(t + 1,x) = f(t,x),\,\forall (t,x) \in {\mathbb R} \times {\mathbb R}.$
In particular, the authors show that the periodic solutions are assembled in mutually ordered connected components and the existence of subharmonic solutions of all periods for \eqref{eq-fo} is achieved, under the hypothesis that a subharmonic solution does exist.
In \cite{ObOm-06} the case of differential inclusions is studied as well, still via upper and lower solutions techniques.\\
Subsequently, the result on the existence of subharmonic solutions of all periods for \eqref{eq-fo} has been also reconsidered
in \cite{AnFuPa-07, Se-09}, employing different approaches. Indeed, in \cite{AnFuPa-07} Andres, F\"urst and Pastor give a proof in terms of multivalued maps, under the additional assumption of global existence for the solutions of \eqref{eq-fo}, while in \cite{Se-09} S\c edziwy exploits direct considerations, based on the geometry of the Euclidean plane.\\
The precise statement of the result proven in \cite{ObOm-07} reads as follows:

\begin{theorem}\label{th-or}
Let $f: {\mathbb R} \times {\mathbb R}\to {\mathbb R}$ be $1$-periodic in the first variable and satisfy the $L^1$-Carath\'eodory conditions on $[0,1]\times\mathbb R.$ If equation \eqref{eq-fo} admits a subharmonic solution of order $m >1,$ then, for every integer $n\ge 1,$ there exists a subharmonic solution of \eqref{eq-fo} of order $n.$
\end{theorem}
 
\noindent 
We recall that in \cite{ObOm-07} it was also shown that the set of all the subharmonic solutions of \eqref{eq-fo} of order $n$ has dimension at least $n$ as a subset of $L^{\infty}(\mathbb R).$ The treatment of such topic is however out of the scope of the present paper, as it doesn't fall within our dynamical approach. Of course, in the statement above the more general case of a map $f$ which is $T$-periodic in the time-variable, for some $T>0,$ could be considered as well. In such a framework, one would assume the existence of a subharmonic solution of order $mT,$ for a certain $m>1,$ and obtain the existence of $nT$-periodic solutions, for every $n\ge 1.$ For the sake of simplicity, we confine ourselves to the setting considered in \cite{ObOm-07}, presenting an elementary verification of Theorem \ref{th-or} based on connectivity. The arguments we employ bear resemblance to the work by S\c edziwy \cite{Se-09}: our proof has however been obtained independently and we present it in full details because it is along the course of such proof that the language and the notations for the subsequent dynamical analysis of the system generated by the solutions of \eqref{eq-fo} are introduced.\\
Our contribution is indeed twofold. On the one hand we propose an alternative dynamical approach to the study of the system under consideration. Namely, since the uniqueness of the solutions is missing, instead of considering the multivalued Poincar\'e operator as in \cite{AnFuPa-07}, we introduce the Bebutov flow, which is defined on a function space. This allows to study the case of differential inclusions as well, without the additional hypothesis of global existence for the solutions of \eqref{eq-fo}. On the other hand, we try to explain which are the chaotic features that can be inferred for the system generated by the solutions of \eqref{eq-fo}. In particular, we are able to show the positivity of the topological entropy and the presence of chaos in the sense of Li-Yorke and Devaney.\\
The paper is organized as follows. In Section \ref{sec-mr} we prove Theorem \ref{th-or}, by splitting its verification into Theorem \ref{th-fs} and the Cancellation Lemma \ref{lem-cl}. More precisely, in Theorem \ref{th-fs} we show that, whenever a subharmonic solution of period two exists for \eqref{eq-fo}, then the presence of periodic solutions of all periods follows. Lemma \ref{lem-cl} states instead that it is always possible to confine ourselves to the case of Theorem \ref{th-fs}, in the sense that, whenever Problem \eqref{eq-fo} admits a subharmonic solution of period $m\ge 3,$ then it also has a solution of period $2.$ As mentioned above, in the proof of Theorem \ref{th-fs} we lay the foundations for the study of the system generated by the solutions of \eqref{eq-fo} in Section \ref{sec-ch}. Here the term ``chaos'' appearing in the title is better specified and the dynamical features of the solutions of Problem \eqref{eq-fo} are more deeply analyzed. In particular we introduce the main tools from the theory of Dynamical Systems we need along the proof of Theorem \ref{th-ch}, where we show that the Bebutov flow and the Bernoulli shift are conjugate. According to \cite{KiSt-89}, this fact has several consequences for the dynamical system generated by the solutions of \eqref{eq-fo}, since it turns out to be Li-Yorke chaotic, sensitive with respect to initial data, topologically transitive, the set of the periodic solutions is dense therein and the topological entropy is positive.\\
The definition of such concepts can be found in Section \ref{sec-ch}, too.

\section{Proof of Theorem \ref{th-or}}\label{sec-mr}

As explained in the Introduction, our proof of Theorem \ref{th-or} depends on two steps. The first one is the verification of Theorem \ref{th-fs} below, which asserts that, whenever a subharmonic solution of period two exists for \eqref{eq-fo}, then the presence of periodic solutions of all periods follows. This is the main part of the proof of Theorem \ref{th-or} and it is here that the language and the notations then used in Section \ref{sec-ch} are introduced. The second step consists instead in showing that it is always possible to confine ourselves with the case considered in Theorem \ref{th-fs}. This is the content of the Cancellation Lemma \ref{lem-cl}, which states that whenever Problem \eqref{eq-fo} admits a subharmonic solution of period $m\ge 3,$ then it also admits a solution of period $2.$ \\
Theorem \ref{th-fs} and Lemma \ref{lem-cl}, as well as their proofs, are presented hereinafter.

\smallskip

\begin{theorem}\label{th-fs}
Let $f: {\mathbb R} \times {\mathbb R}\to {\mathbb R}$ be continuous and $1$-periodic in the time-variable. If equation \eqref{eq-fo} admits a subharmonic solution of order $2,$ then, for every $n\ge 1,$ there exists a subharmonic solution of \eqref{eq-fo} of order $n.$
\end{theorem}
\begin{proof}
Let $x(t)$ be a solution of \eqref{eq-fo} of period two defined for $t\ge t_0,$ for some $t_0\in\mathbb R.$ Then $x(t+2)=x(t),\,\forall t\ge t_0$ and there exists $t_1\ge t_0$ such that $x(t_1+1)\ne x(t_1).$ Without loss of generality we can assume $x(t_1+1)>x(t_1),$ since otherwise it would be sufficient to consider $t_1+1$ in place of $t_1.$ Defining $y(t):=x(t+1),$ we find $y(t_1)=x(t_1+1)>x(t_1)$ and $y(t_1+1)=x(t_1+2)=x(t_1)<x(t_1+1).$ Thus, by Bolzano theorem there exists $\xi\in (t_1,t_1+1)$ such that $y(\xi)=x(\xi).$ Hence, $y(\xi+1)=x(\xi+2)=x(\xi)=y(\xi)=x(\xi+1)$ and, more generally, $y(\xi+n)=x(\xi+n),\,\forall n\in\mathbb N.$ Just to fix ideas, we start by supposing that $\xi$ is the only instant in $(t_1,t_1+1)$ where $x(\cdot)$ and $y(\cdot)$ coincide. By such assumption, on each interval of the form $(\xi+n,\xi+n+1),$ with $n\in\mathbb N,$ it holds that $x(t)>y(t)$ or $x(t)<y(t)$ and the situation gets inverted when moving from $(\xi+n,\xi+n+1)$ to $(\xi+n+1,\xi+n+2)$ (more precisely, $x(t)>y(t)$ on the intervals of the kind $(\xi+2m,\xi+2m+1)$ and $y(t)>x(t)$ on $(\xi+2n+1,\xi+2n+2),$ for $m,n\in\mathbb N$). Thus, in correspondence to every interval of the form $(\xi+n,\xi+n+1),$ we can choose between staying ``up'' or ``down'' by suitably selecting $x(t)$ or $y(t).$ In particular we associate to any such interval the label ``$0$'' when we stay ``up'' and the label ``$1$'' when we stay ``down''. This procedure can obviously be adopted also on the intervals $(\xi+i,\xi+i+1)$ with $i$ negative integer, by extending $x(\cdot)$ and $y(\cdot)$ to the whole real line by $2$-periodicity, thanks to the fact that $f$ is $1$-periodic in the time-variable. In such way we are led to work with the two-sided sequences on two symbols $\eta=(\eta_i)_{i\in\mathbb Z},$ with $\eta_i\in\{0,1\},\,\forall i\in\mathbb Z.$ For any $t\in\mathbb R\setminus\{\xi+i:i\in\mathbb Z\},$ we call $x_0(t)$ the one between $x(t)$ or $y(t)$ that stays ``up'' and $x_1(t)$ the one that stays ``down''. Hence, to any sequence $\eta=(\eta_i)_{i\in\mathbb Z}\in\{0,1\}^{\mathbb Z}$ it is possible to associate the function 
\begin{equation}\label{eq-eta}
w_{\eta}:\mathbb R\to\mathbb R,\,\quad w_{\eta}(t):=x_{\eta_i}(t), \mbox{ for } \xi+i\le t\le \xi+i+1.
\end{equation}
It is easy to check that, setting $p:=x(t_1),\,q:=y(t_1),\,s_1:=t_1+1$ and recalling that $p<q,$ then $w(\cdot)=w_{\eta}(\cdot)$ is a solution of \eqref{eq-fo} which satisfies $w(s_1+i)=q$ if $\eta_i=0$ and $w(s_1+i)=p$ if $\eta_i=1.$ Moreover, it is immediate to see that if $\eta=(\eta_i)_i$ is a periodic sequence of some period $l>1,$ that is, $\eta_{i+l}=\eta_i,\,\forall i\in\mathbb Z,$ then the corresponding solution $w_{\eta}(t)=(x_{\eta_i}(t))_{i\in\mathbb Z}$ is periodic of the same period, i.e. $x_{\eta_i}(t+l)=x_{\eta_i}(t),\,\forall t\in\mathbb R.$ In this way, we have proved the existence of subharmonic solutions of each period for \eqref{eq-fo}. The thesis is achieved.\\
In the more general case in which $x(\cdot)$ and $y(\cdot)$ meet several times in $(t_1,t_1+1),$ let $\xi$ be the first instant in $(t_1,t_1+1)$ such that $x(\xi)=y(\xi).$ Then the same proof presented above still works, with the only difference that the label to assign to the generic interval $(\xi+i,\xi+i+1)$ is now decided by looking at the value that the maps $x(\cdot)$ and $y(\cdot)$ assume in $t_1+i+1.$ Indeed,
$t_1+i+1\in(\xi+i,\xi+i+1),$ for any $i\in\mathbb Z,$ and it holds that $x(t_1+i+1)>y(t_1+i+1)$ when $i$ is even, while $y(t_1+i+1)>x(t_1+i+1)$ when $i$ is odd.
Moreover, for any $t\in\mathbb R\setminus\{\xi+k:k\in\mathbb Z\},$ we have that $t\in (\xi+i,\xi+i+1),$ for a unique $i\in\mathbb Z.$ Then
we call $x_0(t)$ the one between $x(t)$ or $y(t)$ that stays ``up'' in $t_1+i+1$ and $x_1(t)$ the one that stays ``down'' in $t_1+i+1.$ In this way, to any sequence $\eta=(\eta_i)_{i\in\mathbb Z}\in\{0,1\}^{\mathbb Z}$ it is possible to associate the function $w_{\eta}$ as in $\eqref{eq-eta}$ and conclude as before\footnote{Obviously the case in which $x(\cdot)$ and $y(\cdot)$ meet several times in $(t_1,t_1+1)$ could be more deeply analyzed.  
Indeed, when there are $m>1$ intersections $\xi_1,\dots,\xi_m$ between $x(\cdot)$ and $y(\cdot)$ in $(t_1,t_1+1),$ it is possible to work with the sequences on a higher number of symbols, that can be assigned as follows: since in any interval of the form $(\xi_k,\xi_{k+1}),$ with $k=1,\dots,m,$ (where we identify $\xi_{m+1}$ with $\xi_1+1$) we can choose between staying ``up'' or ``down'' by suitably selecting $x(t)$ or $y(t),$ we associate to $(\xi_k,\xi_{k+1})$ the label ``$0$'' when we stay ``up'' and the label ``$1$'' when we stay ``down''. Hence, to describe our $m$ choices on the interval $[\xi_1,\xi_1+1],$ we use a string of $m$ symbols, in which any element can be $0$ or $1.$ Such $2^m$ strings can be ordered lexicographically and each of them may be identified with the integer between $0$ and $2^m-1$ that denotes its position in this order decreased by one. In such way we are led to work with the two-sided sequences on $2^m$ symbols. This allows to infer stronger consequences from a dynamical point of view, as it permits to prove the conjugacy between the space generated by the solutions of \eqref{eq-fo} and the Bernoulli shift on $2^m$ symbols, instead of considering the shift on just two symbols (see Section \ref{sec-ch} for more details). Actually, since in \cite{ObOm-07} the existence of continua of periodic solutions is proven, the presence of chaos on infinite symbols could be shown as well. However, since all the relevant chaotic features are present even with a finite number of symbols, we confine ourselves to the easier framework.}.
\end{proof}
\eop

\begin{lemma}[Cancellation Lemma]\label{lem-cl}
If Problem \eqref{eq-fo} admits a subharmonic solution of period $m,$ with $m\ge 3,$ which is not $1$-periodic, then it also admits a solution of period two, which is not $1$-periodic.
\end{lemma} 
\begin{proof}
Let $u(t)$ be a solution of \eqref{eq-fo} defined for $t\ge t_0,$ for a certain $t_0\in\mathbb R,$ such that there exists $m\ge 3$ with $u(t+m)=u(t),\,\forall t\in\mathbb R,$ and $u(t^*+1)\ne u(t^*),$ for some $t^*\in\mathbb R.$ Then we claim that we can suppose the set $\{u(t^*),u(t^*+1),\dots,u(t^*+m-1)\}$ to be composed by pairwise distinct terms. Indeed, if this were not the case, we could join two of the coinciding elements $u(t^*+j)$ and $u(t^*+k),$ for some $j<k\in\{0,\dots,m-1\},$ in order to obtain an $m-(k-j)$-periodic solution
$\widetilde u$ of \eqref{eq-fo} defined as 
\begin{equation}\label{eq-join}
\widetilde u(t):=\; \left\{
\begin{array}{lll}
u(t)\quad &t\le\,t^*+j\\
u(t+k-j)\quad &t\ge\,t^*+j\,.\\ 
\end{array}
\right.
\end{equation}

\noindent
Since this can be done for any couple of coinciding elements, the claim is true for the solution $\widetilde{\widetilde u}$ of \eqref{eq-fo} so obtained, that we still denote by $u.$\\
If $u$ has period two, the lemma is proved. Otherwise, let us call $\bar t$ the element among $t^*,t^*+1,\dots,t^*+m-1$ such that $u(\bar t)=\min \{u(t^*+i):0\le i\le m-1\},$ so that, setting $v(t):=u(t+m-1),$ we find $v(\bar t)=u(\bar t+m-1)>u(\bar t).$ On the other hand, $v(\bar t+1)=u(\bar t)<u(\bar t+1).$ Hence, by Bolzano theorem there exists $\tilde t\in (\bar t,\bar t+1)$ with $u(\tilde t)=v(\tilde t)=u(\tilde t+m-1)$ and thus we can obtain a $2$-periodic solution of \eqref{eq-fo} with the procedure in \eqref{eq-join}.
The thesis is so achieved.
\end{proof}
\eop

\section{Chaotic Dynamics}\label{sec-ch}

Before stating our main result on chaotic dynamics, i.e. Theorem \ref{th-ch}, we recall the fundamental tools from the theory of dynamical systems we are going to use along its proof. In particular, at first we introduce the Bebutov flow \cite{BhSz-70,Si-75} and then we collect some general definitions and facts about chaotic dynamics. \\
We denote by $\mathcal C$ the set of the continuous functions $f:\mathbb R\to\mathbb R,$ that is,
\begin{equation}\label{eq-c}
\mathcal C=\mathcal C(\mathbb R).
\end{equation}
On this space we define a metric $\rho$ as follows: given an integer $m\ge 1$ and $I_m:=[-m,m],$ for $f,g\in \mathcal C$ we set 
$$\vartheta_m(f,g):=\max\{|f(t)-g(t)|:t\in I_m\},$$
$$\rho_m(f,g):=\frac{\vartheta_m(f,g)}{1+\vartheta_m(f,g)}$$ and
\begin{equation}\label{eq-rho} 
\rho(f,g):=\sum_{m=1}^{\infty}\frac{\rho_m(f,g)}{2^m}.
\end{equation}
One may verify that $\rho$ is indeed a metric on $\mathcal C$ and that with this choice $\mathcal C$ is complete. Moreover, the convergence induced by $\rho$ on $\mathcal C$ is the uniform convergence on compact sets, that is, if $(f_n)_n$ is a sequence in $\mathcal C,$ then $\rho(f_n,f)\to 0$ as $n\to\infty,$ for $f\in\mathcal C,$ if and only if $f_n(t)\to f(t)$ uniformly on compact subsets of $\mathbb R$ \cite{BhSz-70}. \\
On the metric space $(\mathcal C,\rho),$ we define the \textit{Bebutov dynamical system} (or \textit{shift dynamical system} \cite{Si-75}) $\pi:\mathcal C\times \mathbb R\to \mathcal C$ as 
$$\pi(f,t)=g,$$
where $$g(s)=f(t+s),\,\forall s\in\mathbb R.$$
The verification that $\pi$ is a dynamical system can be found in \cite{BhSz-70}. When $s$ is fixed, it is also possible to define the continuous function
\begin{equation}\label{eq-psi}
\psi_s:\mathcal C\to \mathcal C, \quad f(\cdot)\mapsto f(\cdot+s).
\end{equation}
This is the map we will consider, for $s=1$ and restricted to the set $W$ in \eqref{eq-w}, in Theorem \ref{th-ch}.

\smallskip

\noindent
Given an integer $m\ge 2,$ we denote by
$\Sigma_m:=\{0,\dots,m-1\}^{\mathbb Z}$ the set of two-sided sequences of $m$ symbols. Such compact space is usually endowed with the distance
\begin{equation}\label{eq-sd}
\hat d(\textbf{s}', \textbf{s}'') := \sum_{i\in {\mathbb Z}} \frac{|s'_i - s''_i|}{m^{|i| + 1}}\,,\quad
\mbox{ for }\; \textbf{s}'=(s'_i)_{i\in {\mathbb Z}}\,,\;
\textbf{s}''=(s''_i)_{i\in {\mathbb Z}}.
\end{equation}
Here we define the \textit{two-sided Bernoulli shift} on $m$ symbols\footnote{We have chosen to present this definition in the generic case of $m\ge 2$ symbols because of the discussion along the previous footnote. However, in view of Theorem \ref{th-ch}, from now on we will confine ourselves to the special framework of two symbols. In particular, this holds true for the definition of chaos in the coin-tossing sense, that we directly give in the less general version, but that could as well be formulated for an arbitrary number of symbols.} $\sigma: \Sigma_m\to \Sigma_m$ as $\sigma((s_i)_i):=(s_{i+1})_i,$ $\forall i\in\mathbb Z.$ Observe that, by the choice of the metric, this map is continuous (and hence a homeomorphism).\\
Given two continuous self-maps $f:Y\to Y$ and $g:Z\to Z$ of the metric spaces $Y$ and $Z,$ we say that $f$ and $g$ are \textit{topologically conjugate}\footnote{Notice that the concept of topological conjugacy can also be defined in the more general setting of topological spaces. However, in order to make the presentation more uniform, we have decided to introduce all the notions in the context of metric spaces. For instance, this remark applies to the definition of coin-tossing chaos.} if there exists a homeomorphism $\phi:Y\to Z$ that makes the diagram
\begin{equation}\label{diag-comm}
\begin{diagram}
\node{{Y}} \arrow{e,t}{f} \arrow{s,l}{\phi}
      \node{{Y}} \arrow{s,r}{\phi} \\
\node{Z} \arrow{e,b}{g}
   \node{Z}
\end{diagram}
\end{equation}
\textit{commute}, i.e. such that $\phi\circ f=g\circ\phi.$ Any such map $\phi$ is named \textit{conjugacy}.\\
A self-map $f: X \to X$
of the metric space $X$ is called
\textit{chaotic in the sense of coin-tossing} \cite{KiSt-89}
if there exist two nonempty disjoint
compact sets
$${\mathcal K}_0,\, {\mathcal K}_1\subseteq X,$$
such that, for each two-sided sequence $(s_i)_{i\in {\mathbb Z}}
\in  \Sigma_2,$ there exists a corresponding sequence
$(x_i)_{i\in {\mathbb Z}}\in {X}^{\mathbb Z}$ such that
\begin{equation}\label{eq-ch}
x_i \,\in\, {\mathcal K}_{s_i}\;\;\mbox{ and }\;\, x_{i+1} =
f(x_i),\;\; \forall\, i\in {\mathbb Z}.
\end{equation}
Given a self-map $f:X\to X$ of the metric space $X,$
we say that $S\subseteq X$ is a \textit{scrambled set for $f$}
if for any $x,y\in S,$ with $x\ne y,$ it holds that
$$ \liminf_{n\to\infty}d_X(f^n(x), f^n(y))=0 \quad \mbox{and} \quad \limsup_{n\to\infty}d_X(f^n(x), f^n(y))>0.$$
If the set $S$ is uncountable, we say that $f$ is \textit{chaotic in the sense of Li-Yorke}.\\
Finally, a self-map $f:X\to X$ of the infinite metric space $X$ is called \textit{chaotic in the sense of Devaney} if:
\begin{itemize}
\item  $f$ is \textit{topologically transitive}, i.e. for any
couple of nonempty open subsets $U,\,V$ of $X$ there exists an
integer $n\ge 1$ such that $U\cap f^n(V)\ne\emptyset\,;$ 
\item the set of the periodic points for $f$ is dense in $X$ \footnote{In the original definition of Devaney chaos \cite{De-89}, it was also required the map $f$ to be
\textit{sensitive with respect to initial data on} $X,$ i.e.
there exists $\delta > 0$ such that for any $x\in X$ there is a
sequence $(x_i)_{i\in\mathbb N}$ in $X$ such that
$x_i\to x$ when $i\to\infty$ and
for each $i\in\mathbb N$ there exists a positive integer
$m_i$ with $d_X(f^{m_i}(x_i), f^{m_i}(x))\ge\delta\,.$ This third condition has however been proved in \cite{BaBr-92} to be redundant for any continuous self-map of an infinite metric space, as it is implied by the previous two.}.
\end{itemize}

\smallskip

\noindent
In Theorem \ref{th-ch} below we are going to show that the map $\psi_1$ in \eqref{eq-psi} (restricted to a suitable compact set) and the Bernoulli shift $\sigma$ are conjugate. From this fact, many chaotic features of the Bernoulli system can be directly transferred to the Bebutov flow by using the next result from \cite{KiSt-89}, that we recall for the reader's convenience, rewritten in conformity with our notation and limited to what is indeed needed along the proof of Theorem \ref{th-ch}.

\begin{lemma}[Kirchgraber \& Stoffer {\cite[Proposition 1]{KiSt-89}}]\label{pr-ks}
Let $f:X\to X$ and $g:Y\to Y$ be homeomorphisms of the complete metric spaces $X$ and $Y,$ respectively. Assume that $f$ and $g$ are topologically conjugate and that $X$ is compact. Then, if $f$ is chaotic in the sense of coin-tossing, in the sense of Devaney, in the sense of Li-Yorke, so is $g.$
\end{lemma}

\noindent
We are now in position to prove the following:
\begin{theorem}\label{th-ch}
Recalling the definition of $w_{\eta}$ in \eqref{eq-eta} and $\mathcal C$ in \eqref{eq-c}, let
\begin{equation}\label{eq-w}
W:=\{w_{\eta}:\eta=(\eta_i)_{i\in\mathbb Z}\in \Sigma_2\}\subset \mathcal C.
\end{equation}
Then $W$ is compact, the map $\psi_1$ in \eqref{eq-psi} restricted to $W$ is a homeomorphism and there exists a conjugacy $\varphi$ between the two-sided Bernoulli shift $\sigma$ on two symbols and $\psi=\psi_1|_W.$
As a consequence, $\psi$ displays the following chaotic features:
\begin{itemize}
\item[$(i)$] $h_{\rm top}(\psi)= h_{\rm top}(\sigma) = \log(2),$
where $h_{\rm top}$ is the topological entropy;

\item[$(ii)$] the map $\psi$ is chaotic in the sense of coin-tossing, Li-Yorke and Devaney; 

\item[$(iii)$] As regards the coin-tossing chaoticity, the map $\psi$ actually displays a stronger property, as the periodic sequences of symbols in $\Sigma_2$ get realized by periodic orbits of $\psi.$ In symbols, this means that whenever $(s_i)_{i\in {\mathbb Z}}\in \Sigma_2$ is a $k$-periodic
sequence (that is, $s_{i+k} = s_i\,,\forall i\in {\mathbb Z}$) for
some $k\geq 1,$ then there exists a corresponding $k$-periodic sequence
$(w_{\eta}^{(i)})_{i\in {\mathbb Z}}\in {W}^{\mathbb Z}$ satisfying
$$w_{\eta}^{(i)} \,\in\, {\mathcal K}_{s_i}\;\;\mbox{ and }\;\, w_{\eta}^{(i+1)} =
\psi(w_{\eta}^{(i)}),\;\; \forall\, i\in {\mathbb Z},$$
where $\mathcal K_i=\{w_{\eta}\in W: \eta_0=i\},\,i=0,1,$ are disjoint and compact.
\end{itemize}
\end{theorem}
\begin{proof}
In order to show that there exists a conjugacy $\varphi$ between $\sigma$ and $\psi,$ let us first check that $\psi(W)\subseteq W.$ Notice that, applying $\psi$ to $w_{\eta}(t)\in W,$ we get $\psi(w_{\eta}(t))=w_{\eta}(t+1)=w_{\sigma(\eta)}(t)$ and this is an element of $W,$ since $\sigma(\eta)\in\Sigma_2.$
Let us now define $\varphi:\Sigma_2\to W$ in the natural way, i.e. as $\varphi(\eta)=w_{\eta}(\cdot).$ Then the diagram  
\begin{equation}\label{diag-comm2}
\begin{diagram}
\node{\Sigma_2} \arrow{e,t}{\sigma} \arrow{s,l}{\varphi}
      \node{\Sigma_2} \arrow{s,r}{\varphi} \\
\node{W} \arrow{e,b}{\psi}
   \node{W}
\end{diagram}
\end{equation}
commutes, since $\varphi(\sigma(\eta))=w_{\sigma(\eta)}(\cdot)=w_{\eta}(\cdot\,+1)=\psi(w_{\eta}(\cdot))=\psi(\varphi(\eta)).$
The fact that $\varphi$ is a bijection directly follows from its definition. The continuity (in fact, uniform continuity) of $\varphi$ comes from the choice of the distances $\hat d$ on $\Sigma_2$ and $\rho$ on $W$ according to \eqref{eq-sd} and \eqref{eq-rho}, respectively. Indeed, given an arbitrary $\varepsilon>0$ we have to find a $\delta>0$ such that, for any $\eta=(\eta_i)_i,\nu=(\nu_i)_i\in \Sigma_2$ with $\hat d(\eta,\nu)<\delta,$ then $\rho(w_{\eta},w_{\nu})<\varepsilon.$ To such aim, let us fix an integer $m>>0$ so that $1/2^m<\varepsilon$ and observe that, in order to have $\rho(w_{\eta},w_{\nu})<\varepsilon,$ it is sufficient to prove that $\vartheta_m=0,$ i.e. $w_{\eta}\equiv w_{\nu}$ on $I_m=[-m,m].$ Namely, if this is the case, $\rho(w_{\eta},w_{\nu})\le 1/2^m<\varepsilon.$ 
Let $m'$ be a positive integer such that $[\xi-m',\xi+m']\supseteq [-m,m].$ Choosing $\delta=1/2^{m'+1},$ we have that $\hat d(\eta,\nu)<\delta$ implies $\eta_i=\nu_i,\,\forall |i|\le m'.$ Hence, $w_{\eta}\equiv w_{\nu}$ holds on $[\xi-m',\xi+m']\supseteq [-m,m]$ \footnote{Actually, $w_{\eta}\equiv w_{\nu}$ on $[\xi-m',\xi+m'+1].$} and thus $\rho(w_{\eta},w_{\nu})<\varepsilon.$ The continuity of $\varphi^{-1}$ follows from the fact that $\varphi$ is a continuous bijection between the compact set $\Sigma_2$ and the Hausdorff space $W.$ Notice that, by the continuity of $\varphi$ and the compactness of $\Sigma_2,$ the set $W=\varphi(\Sigma_2)$ is compact, too, and thus complete.\\
By the conjugacy between $\sigma$ and $\psi,$ we obtain $h_{\rm top}(\psi)= h_{\rm top}(\sigma)$ (see \cite[Theorem 7.12]{Wa-82}). On the other hand, it is a well known fact that $h_{\rm top}(\sigma)=\log(2)$ and this concludes the proof of $(i).$\\
As regards $(ii),$ according to \cite{KiSt-89}, the Bernoulli system is chaotic in the sense of coin-tossing, Devaney and Li-Yorke.
In order to apply Lemma \ref{pr-ks}, we only have to check that $\psi:W\to W$ is a homeomorphism. By its definition and the choice of $W,$ it is immediate to see that it is a bijection on $W$ with inverse $\psi^{-1}(w_{\eta})=w_{\sigma^{-1}(\eta)}$ and again the continuity of $\psi^{-1}$ comes from the fact that $\psi$ is a continuous bijection between the compact set $W$ and the Hausdorff space $W.$ Thus Lemma \ref{pr-ks} allows to reach conclusion $(ii)$ \footnote{We stress that in \cite{KiSt-89} the definition of Li-Yorke chaos requires, besides the classical assumptions \cite{LiYo-75}, the scrambled set to be invariant. Since the Bernoulli shift displays this stronger property \cite{KiSt-89}, by \cite[Proposition 1]{KiSt-89} the same is true also for $\psi.$}.\\
Finally, $(iii)$ directly follows from the definition of the map $\varphi$ which indeed maps periodic sequences of symbols into periodic orbits of $\psi.$\\
This concludes the proof.
\end{proof}
\eop

\section{Acknowledgments}
Many thanks to Professor Zanolin for suggesting me this interesting problem and
for his invaluable help during the preparation of the paper.

\end{document}